\documentclass[11pt,a4paper]{article}   %% LaTeX2e document.

\usepackage{amssymb, amsmath, bbm, amsthm}
\usepackage{fullpage}
\usepackage{graphicx}
\usepackage{url}

\theoremstyle{plain}
\newtheorem{theorem}{Theorem}[section]
\newtheorem{lemma}[theorem]{Lemma}
\newtheorem{corollary}[theorem]{Corollary}
\newtheorem{proposition}[theorem]{Proposition}
\newtheorem{example}[theorem]{Example}
\numberwithin{equation}{section}
\theoremstyle{definition}
\newtheorem{definition}[theorem]{Definition}
\newtheorem{assumption}[theorem]{Assumption}
\theoremstyle{remark}
\newtheorem{remark}[theorem]{Remark}

\newcommand{\R}{{\mathbb R}}

\newcommand{\intr}{\mathrm{int}}

\newcommand{\F}{{\mathbb F}}

\newcommand{\thstar}{\mathop{\theta^*}}
\newcommand{\zstar}{\mathop{z^*}}

\newcommand{\xstar}{\mathop{x^*}}
\newcommand{\Xstar}{\mathop{X^*}}

\newcommand{\E}{\mathop{\mathbb{E}}\nolimits}

\newcommand{\hth}{\hat{\theta}}

\newcommand{\hx}{\hat{x}}

\newcommand{\armax}{\mbox{argmax}}

\newcommand{\inte}{\mbox{int}}
\newcommand{\supp}{\mbox{supp}}

\newcommand{\Keywords}[1]{\par\noindent 
{\small{\em Keywords\/}: #1}}
\newcommand{\AMS}[1]{\par\noindent 
{\small{\bf AMS MSC 2010\/}: #1}}

\makeatletter
\def\Ddots{\mathinner{\mkern1mu\raise\p@
\vbox{\kern7\p@\hbox{.}}\mkern2mu
\raise4\p@\hbox{.}\mkern2mu\raise7\p@\hbox{.}\mkern1mu}}
\makeatother
\title{\Large Parameter dependent optimal thresholds, indifference levels and inverse optimal stopping problems}
\author{Martin Klimmek\thanks{Mathematical Institute, University of Oxford, Oxford OX13LB \newline email: Martin.Klimmek@maths.ox.ac.uk}  \\
\small }
\date{\small \today}

\begin{document}
 \maketitle
\begin{abstract}
Consider the classic infinite-horizon problem of stopping a one-dimensional diffusion to optimise between running and terminal rewards and suppose we are given a parametrised family of such problems. We provide a general theory of parameter dependence in infinite-horizon stopping problems for which threshold strategies are optimal. The crux of the approach is a supermodularity condition which guarantees that the family of problems is indexable by a set valued map which we call the indifference map. This map is a natural generalisation of the allocation (Gittins) index, a classical quantity in the theory of dynamic allocation. Importantly, the notion of indexability leads to a framework for inverse optimal stopping problems.
\smallskip
\AMS{60G40; 60J60 }
\end{abstract}

\section{Introduction}
Consider the following classical optimal stopping problem. Given a discount parameter and a time-homogeneous diffusion started at a fixed point, we are asked to maximise an expected payoff which is the sum of a discounted running reward up until the stopping time and a terminal reward depending on the state of the diffusion at the stopping time. We call this problem the forward optimal stopping problem and the expected payoff under the optimal stopping rule the (forward) problem value. 

The problem can be generalised to a parametrised family of reward functions to give a parametrised family of forward problems. This generalisation is often natural. For instance, in economics we may be interested in the effect of changes in a dividend or a tax rate on the value of an investment and the optimal investment decision. In dynamic resource allocation problems, a parameter may act an index for different projects. In this context, the decision of which  project to engage requires an analysis of the parameter dependence of optimal stopping rules and problem values. 

The approach to solving forward problems in this article is motivated by previous work for the case when there is no running reward. In the case of perpetual American puts, Ekstr{\"o}m and Hobson \cite{HobsonEkstroem} establish convex duality relations between value functions and the Laplace transform of first hitting times of the underlying diffusion. In related work by Lu \cite{BingLu}, the approach is developed to establish duality when the parameter space is a discrete set of strikes. More generally, Hobson and Klimmek \cite{HobsonKlimmek:10} employ generalized convex analysis to establish duality between $\log$-transformed value functions and $\log$-transformed diffusion eigenfunctions for a general class of reward functions. The common strand in this previous work on inverse stopping problems is the conversion of a stochastic problem into a deterministic duality relation involving monotone optimizers. 

This article provides a unifying view of the monotone comparative statics results for optimal stopping developed previously and an extension to non-zero running rewards. We show that a supermodularity condition on the reward functions guarantees monotonicity of optimal thresholds in the parameter value. This monotonicity of the thresholds imposes a useful and natural order on families of parametrised stopping problems through a generalisation of the so-called allocation (or Gittins) index, an important quantity in the theory of dynamic allocation problems (see for instance Whittle \cite{Whittle} and Karatzas \cite{Karatzas}). We utilise the notion of indexability to solve parametrised families of stopping problems. 

As well as solving families of forward problems, we consider the problem of recovering diffusion processes consistent with given optimal stopping values.  `Inverse optimal stopping problems' find natural motivation in mathematical finance and economics. When there is no running reward, the problem has the interpretation of constructing models for an asset price process consistent with given perpetual American option prices. Now suppose instead that we are given an investor's valuation for a dividend bearing stock which may be liquidated for taxed capital gains. Given the valuation, we would like to recover the investor's model. Similar situations may arise in a real-options setup. A bidder for a resource extraction project may submit a range of bids for a project depending on an economic parameter. In this case, a regulator might naturally be interested in recovering the investor's model which underlies the bids. This article provides solutions to inverse problems in the presence of a non-zero running reward (or cost). We show that the value function does not contain enough information to recover a diffusion and that solutions to the inverse problem are parametrised by a choice of indifference (allocation) index. The indifference index can be interpreted as representing an investor's preferences with respect to remaining invested or liquidating. Given consistent preferences and valuations, it is possible to recover a diffusion model.

%In the setting of investment problems, generalisations of the allocation index have a natural interpretation of representing investor preferences.  While we show that it is in general not possible to recover diffusion models for a risky asset consistent with a value function for an investment, if we are also given investor preferences in the form of an allocation index, then it is possible to recover consistent diffusions which are uniquely specified on a subset of their state space. 
This article provides a direct approach to forward and inverse problems based on principles from monotone comparative statics and dynamic allocation. In spirit, the direct approach is related to recent seminal work by Dayanik and Karatzas \cite{DayanikKaratzas} and Bank and Baumgarten \cite{Bank}. The direct solution method in \cite{DayanikKaratzas}, based on the calculation of concave envelopes, is employed by Bank and Baumgarten \cite{Bank} to solve parameter-dependent forward problems. However, the method used in \cite{Bank} is restricted to problems with linear parameter dependence and requires calculation of an auxiliary function which transforms general two-sided stopping problems to one-sided threshold problems. The approach taken in this article is to focus on optimal stopping problems for which one-sided threshold strategies are optimal. This restriction (which is usual in the setting of dynamic allocation problems) leads to a tractable characterization of parameter-dependence. As an analysis of 
allocation indices and stopping problems, this article can be seen to extend the work of Karatzas \cite{Karatzas}. However, the aim here is not to prove the optimality of the `play-the-leader' policy for multi-armed bandits, but to generalise the approach to inverse optimal stopping problems introduced in \cite{HobsonKlimmek:10}. The fundamental aim is to establish qualitative principles that govern the relationship between data (e.g. prices), economic behaviour (e.g. investment indifference levels) and models (e.g. generalised diffusions). 

%Previous related work includes Alvarez \cite{Alvarez}, Dayanik and 
%Karatzas \cite{DayanikKaratzas} and Bank \cite{Bank}. Alvarez solves the type of stopping problem we have described in a (non-parameter dependent) setting where the diffusion co-efficients are smooth and optimal stopping times are first exit times from an interval by optimizing a non-linear program. 
%The main motivation for considering forward problems in \cite{HobsonKlimmek10} is to establish techniques to solve so-called `inverse' problems. If value functions are given for a range of parameters, the inverse problem is to characterise sets of diffusions which are consistent with the given values. In mathematical finance this type of problem is motivated by the idea that option prices can be observed for a range of strikes. The issue is to find models that fit (are consistent with) the observed prices. 
%Inverse problems are complicated by the addition of a running reward. 

\section{Forward and the inverse problems}
Let $X=(X_t)_{t \geq 0}$ be a diffusion process on an interval $I$, let $\rho$ be a discount parameter. Let $G = \{G(x,\theta);\theta \in \Theta\}$ be a family of terminal reward functions and $c= \{c(x,\theta); \theta \in \Theta\}$ a family of running reward functions, both parametrised by a real parameter $\theta$ lying in an interval $\Theta$ with end-points $\theta_-$ and $\theta_+$. The classical approach in optimal stopping problems is to fix the parameter, i.e. $\Theta=\{\theta\}$, and calculate \[V(x)=\sup_{\tau}\E_{x} \left[\int_0^\tau e^{-\rho t} c(X_t,\theta) dt + e^{-\rho \tau} G(X_\tau,\theta) \right]\]
for $x \in \intr(I)$ using variational techniques, see for instance 
Bensoussan and Lions \cite{Lions}.

In contrast, we are interested in the case when the starting value is fixed and the parameter varies. Then the {\it forward problem} is to calculate $V \equiv \{V(\theta) \ ; \ \theta \in \Theta\}$ where 
\begin{equation} \label{eq:forward}
V(\theta)=\sup_{\tau} \E_{X_0}\left[\int_0^\tau e^{-\rho t} c(X_t,\theta) dt + e^{-\rho \tau} G(X_\tau,\theta) \right].
\end{equation}

We will assume that the process underlying the stopping problem is a regular one-dimensional diffusion processes characterised by a speed 
measure and a strictly increasing and continuous scale function. Such diffusions are `generalised' because the speed measure need not have a density. 

Let $I \subseteq \R$ be a finite or infinite interval with a left endpoint $a$ and right endpoint $b$. Let $m$ be a non-negative, non-zero Borel measure on $\R$ with $I=\supp(m)$. Let $s: I \rightarrow \R$ be a strictly increasing and continuous function. Let $x_0 \in I$ and let $B=(B_t)_{t \geq 0}$ be a Brownian motion started at $B_0=s(x_0)$ supported on a filtration $\F^B=({\mathcal F}_u^B)_{u\geq 0}$ with local time process
$\{ L_u^z ; u \geq 0, z \in \R \}$. Define $\Gamma$ to be the
continuous, increasing, additive functional
\[\Gamma_u = \int_{\R} L_u^z m(dz),\]
and define its right-continuous inverse by
\[A_t = \inf \{u : \Gamma_u > t \}. \]
If $X_t = s^{-1}(B(A_t))$ then $X=(X_t)_{t \geq 0}$ is a one-dimensional regular diffusion started at $x_0$ with speed measure $m$ and scale function $s$. Moreover, $X_t \in I$ almost surely for all $t \geq 0$. 

Let $H_x=\inf\{u:X_u=x\}$. Then for a fixed $\rho>0$ (see e.g. \cite{Salminen}),
\begin{equation} \label{eq:eigenfunction}
\xi(x,y)=\E_{x}[e^{-\rho H_y}]= \left\{\begin{array}{ll}
\frac{\varphi(x)}{\varphi(y)}  &\; x \leq y \\
\frac{\phi(x)}{\phi(y)} &\; x \geq y ,
\end{array}\right.
\end{equation}
where $\varphi$ and $\phi$ are respectively a strictly increasing and a strictly decreasing solution to the differential equation
\begin{equation} \label{eq:differentialsc}
\frac{1}{2} \frac{d}{dm} \frac{d}{ds} f = \rho f.
\end{equation}
In the smooth case, when $m$ has a density $\nu$ so that $m(dx)= \nu(x)dx$ and $s''$ is continuous,
(\ref{eq:differentialsc}) is equivalent to
\begin{equation} \label{eq:differentialsnice}
\frac{1}{2} \sigma^2(x) f''(x) + \alpha(x) f'(x) = \rho f(x),
\end{equation}
where \[ \nu (x)=  \sigma^{-2}(x) e^{M(x)}, \ \ s'(x)=e^{-M(x)}, \  \
M(x)=\int_{0-}^x 2 \sigma^{-2} (z) \alpha(z) dz. \]

We will call the solutions to (\ref{eq:differentialsc}) the $\lambda$-eigenfunctions of the diffusion. For a fixed diffusion with a fixed starting point we will scale $\varphi$ and $\phi$ so that $\varphi(X_0)=\phi(X_0)=1$. The boundary conditions of the differential equation (\ref{eq:differentialsc}) depend on whether the end-points of $I$ are inaccessible, absorbing or reflecting, see Borodin and Salminen \cite{borodin} for details. We will denote by $\inte(I)$ the interior of $I$ and its accessible boundary points and we will make the following  assumption about the boundary behaviour of $X$. 

\begin{assumption}
Either the boundary of $I$ is non-reflecting (absorbing or killing) or
$X$ is started at a reflecting end-point and the other end-point is 
non-reflecting.
\end{assumption} 

Now, for $\theta \in \Theta$, let 
\begin{equation} \label{d:resolvent}
R(x,\theta)=\E_x \left[\int_0^\infty e^{-\rho t} c(X_t,\theta) dt \right].
\end{equation}
Define $U:I\times \Theta \rightarrow \R$ by  $U(x,\theta)=G(x,\theta)-R(x,\theta)$ and for all $\theta \in \Theta$ and $x \in I$ let $c^\theta(x)=c(x,\theta)$ and $R^\theta(x)=R(x,\theta)$.  

\begin{assumption}
$\E_x \left[\int_0^\infty e^{-\rho s} |c^\theta(X_s)| ds  \right] < \infty$ for all $x \in \intr(I)$ and $\theta \in \Theta$.
\end{assumption}
Under our assumptions it is well-known (see for instance Alvarez \cite{Alvarez}) that $R^\theta: \intr(I) \rightarrow \R$ solves the differential equation 
\begin{equation} \label{eq:runningr}
\frac{1}{2} \frac{d}{dm} \frac{d}{ds} f = \rho f - c^\theta.
\end{equation}

\begin{example} \label{ex:ResGeom}
In some cases $R^\theta$ can be calculated directly. Let $\mu < \rho$ and let $dX_t=\sigma X_t dB_t + \mu X_t dt$ and $c(x,\theta)=x \theta$. Then 
$\E_x \left[\int_0^\infty e^{-\rho t} X_t \theta dt \right] dt=x \theta \int_0^\infty e^{(\mu-\rho)t} dt=\frac{x\theta}{\rho-\mu}$.
\end{example}

\begin{example} \label{ex:ResBessel}
Suppose $m(dx)=2x^2 dx$ and $s(x)=-1/x$. Then $X$ is known as the three-dimensional Bessel process and solves the SDE; $dX_t=dB_t+dt/X_t$. Let $c:\R^2 \rightarrow \R$ be defined $c(x,\theta)=\theta \cos(x)$ and $\rho=1/2$. Then $R^\theta$ solves $\frac{1}{2} f''(x)+f'(x)/x-\frac{1}{2} f(x)=-\theta \cos(x)$ with $f(0)=0$. The solution is 
$R^\theta(x)=\theta\left(\cos(x)-\frac{\sin(x)}{x}\right)$.
\end{example}

In order to rule out the case of negative value functions we also make the following assumption. 
\begin{assumption} \label{ass:Gcontinuous}
For all $\theta \in \Theta$, $x \rightarrow U(x,\theta)$ there exists $\hat{x} \in \intr(I)$ such that $U(\hat{x},\theta)>0$. 
\end{assumption}

\subsection{Summary of the main results}
Our main result for the forward problem can be summarised as follows.

{{\bf Solution to the forward problem:} {\it Given a generalised diffusion $X$, if $U(x,\theta)=G(x,\theta)-R(x,\theta)$ is $\log$-supermodular then a threshold strategy is optimal on an interval $(\theta_-, \theta_R)$ and an optimal finite stopping rule does not exist for $\theta > \theta_R$. Furthermore, if $U$ is sufficiently regular and $V$ is differentiable at $\theta \in (\theta_-,\theta_R)$ then \[V'(\theta)=\frac{U_\theta(\xstar(\theta),\theta)}{\varphi(\xstar(\theta))},\] where $\xstar: \Theta \rightarrow I$ is a monotone increasing function such that $\tau=H_{\xstar(\theta)}$ is the optimal stopping rule.}}

Now suppose that we are given $V=\{V(\theta) \ ; \ \theta \in \Theta\}$ and $G=\{G(x,\theta) \ ; \ x \in \R, \  \theta \in \Theta\}$, $c=\{c(x) \ ; \ x \in \R\}$ and $X_0$. Then the {\it inverse problem} is to construct a diffusion $X$ such that $V_X=V$ is the value function corresponding to an optimal threshold strategy. (To keep the inverse problem tractable we focus on the case when the running cost is not parameter dependent.) Our analysis hinges on specifying the parameters for which it is optimal to stop immediately (i.e. $\tau=0$) for a given level of the underlying diffusion. If we consider $V$ to be the value of an investment as a function of a parameter (e.g. a level of capital gains tax), then the indifference map specifies the parameters for which an investor would be indifferent whether to invest or not as it would be optimal to sell immediately.

The indifference map is a natural extension of the allocation (Gittins) index which occurs naturally in the theory of multi-armed bandits. We provide a novel application of this classical quantity in the context of inverse investment problems and real option theory. The indifference map can be seen to represent investor preferences with respect to liquidating for capital gains or remaining invested for future returns. Depending on the valuation of an investment as a function of the parameter, we will show how to recover diffusion models for the underlying risky asset consistent with given preferences (indifference maps).

{{\bf Solution to the inverse problem:} {\it Solutions to the inverse problem are parametrised by a choice of allocation index $\thstar: I \rightarrow \Theta$: The functions $\varphi$ and $R$ defined 
\[\varphi(x)= \frac{G_\theta(x,\thstar(x))}{V'(\thstar(x))}, \ \ \ R(x)= G(x,\thstar(x))-\varphi(x) V(\thstar(x)),\]
determine the speed measure and scale function of the solution through equations (\ref{eq:differentialsc}) and (\ref{eq:runningr}). 
}}

\section{The forward problem: threshold strategies}
Threshold strategies are a natural class of candidates for the optimal stopping time in the forward problem. Our first aim is to establish necessary and sufficient conditions for the optimality of a threshold strategy.

By the strong Markov property of one-dimensional diffusions the value function for the optimal stopping problem can be decomposed into the reward from running the diffusion forever and an {\it early stopping reward}.
\begin{equation} \label{eq:strongmarkov}
V(x,\theta)=R(x,\theta)+\sup_\tau \E_{x}[e^{-\rho \tau}(G(X_\tau,\theta)-R(X_\tau,\theta))].
\end{equation}
We will let $E(x,\theta)=V(x,\theta)-R(x,\theta)$ denote the {\it optimal early stopping reward} and let $U(x,\theta)=G(x,\theta)-R(x,\theta)$ denote the {\it early stopping reward function}. 
%\begin{definition}
%The {\it threshold strategy} for a parameter $\theta \in \Theta$ is the set of points $\Xstar(\theta)$ such that for all $x \in \Xstar(\theta)$, $\tau=H_x$ is an optimal stopping time for the forward problem (\ref{eq:forward}). %with parameter $\theta \in \Theta$. 
%\end{definition}

\begin{lemma} \label{l:oneside}
Stopping at the first hitting time of $z \geq X_0$, $z \in \intr(I)$ is optimal if and only if $\frac{U(y,\theta)}{\varphi(y)}$ attains its global maximum on $\intr(I)$ at $z$.
\end{lemma} 

\begin{proof}
Suppose that the global maximum is achieved at $z \geq X_0$. Let 
\begin{equation*}
\hat{E}(\theta)= \frac{U(z,\theta)}{\varphi(z)}.
\end{equation*}
We will show that $E(X_0,\theta)=\hat{E}(\theta)$. On the one hand, $E(X_0,\theta) \geq \hat{E}(\theta)$ since the supremum over all stopping times is larger than the value of stopping upon hitting a given threshold. Moreover $e^{-\rho t}\varphi(X_t)$ is a non-negative local martingale hence a super-martingale. We have that for all stopping times $\tau$,
\[1 \geq \E_{X_0}\left[e^{-\rho \tau} \varphi(X_{\tau})\right] \geq \E_{X_0} \left[e^{-\rho \tau} \frac{U(X_\tau,\theta)}{\hat{E}(\theta)}\right],\]
and hence $\hat{E}(\theta) \geq \E_{X_0}[e^{-\rho \tau}(G(X_\tau,\theta)-R(X_\tau,\theta))]$ for all stopping times $\tau$. Hence $H_z$ is optimal.

For the converse, suppose that there exists an $z' \in \intr(I)$, $z' \neq z$ such that $\frac{U(z',\theta)}{\varphi(z')} >\frac{U(z,\theta)}{\varphi(z)}$. We will show that there exists a stopping time which is better than $H_z$. First,
if $z' \geq X_0$ then stopping at $\tau =H_{z'}$ is a better strategy than stopping at $\tau=H_z$. Now suppose $z'< X_0$. Then
\begin{eqnarray*}
 U(z,\theta) \E_{X_0}[e^{-\rho H_z}]&=& U(z,\theta)\E_{X_0}[e^{-\rho H_z} 1_{H_z < H_{z'}}] + U(z,\theta) \E_{X_0}[e^{-\rho H_{z'}} 1_{H_{z'}<H_z}] \E_{z'}[e^{-\rho H_z}] \\
&=& U(z,\theta)\E_{X_0}[e^{-\rho H_z} 1_{H_z < H_{z'}}] +  U(z',\theta) \E_{X_0}[e^{-\rho H_{z'}} 1_{H_{z'}<H_z}] \frac{U(z,\theta)/\varphi(z)}{U(z',\theta)/\varphi(z')} \\
&<& U(z,\theta)\E_{X_0}[e^{-\rho H_z} 1_{H_z < H_{z'}}] +  U(z',\theta) \E_{X_0}[e^{-\rho H_{z'}} 1_{H_{z'}<H_z}] ,
\end{eqnarray*}
so stopping at $H_{(z',z)}$ is better than stopping at $H_z$. 
\end{proof}

\begin{remark} \label{r:parallel}
There is a parallel result for stopping at a threshold below $X_0$. A threshold below $X_0$ is optimal if and only if $\frac{U}{\phi}$ attains a global maximum below $X_0$.
\end{remark}

\begin{example} \label{ex:fullgeo}
Recall Example \ref{ex:ResGeom} and let $X$ be a Geometric Brownian Motion started at $1$ with volatility parameter $\sigma$ and drift parameter $\mu<\rho$. Suppose $\Theta=\R^+$, $G(\theta)=\theta$ and $c(x,\theta)=x$. Then $U(x,\theta)=G(x,\theta)-R(x,\theta)=\theta-x/(\rho-\mu)$. %If $\rho < \mu$ then it is never optimal to stop and $V(\theta)=\infty$. Suppose that $\rho>\mu$. 
$U(x,\theta)$ is decreasing so we look for a stopping threshold below $1$. $\phi(x)=x^{-\sqrt{\nu^2+2\rho/\sigma^2}-\nu}$ for $0 < x \leq 1$, where $\nu=\mu/\sigma^2-1/2$. Let $c_-=\sqrt{\nu^2+2\rho/\sigma^2}+\nu$ and $x(\theta)=\frac{c_- \theta (\rho-\mu)}{1+c_-}$. If $0 < x(\theta) \leq 1$ then $x(\theta)$ is the optimal stopping threshold. If $x(\theta)=0$ then it is optimal to `wait forever'. If $x(\theta)>1$ then it is optimal to stop immediately. 
\end{example}

The following Lemma shows that if a threshold strategy is optimal then the optimal threshold is either above or below the starting point. This rules out the case that both an upper threshold and a lower threshold are optimal for a fixed parameter.

\begin{lemma} \label{l:triangle}
For a fixed parameter $\theta$, let $U(s)=U(s,\theta)$. Let $\displaystyle \triangle_-=\{z:z \in \armax_s [U(s)/\phi(s)]\}$ and 
$\displaystyle \triangle_+=\{z:z \in \armax_s [U(s)/\varphi(s)]\}$. If 
$x \in \triangle_+$ and $y \in \triangle_-$ then 
$x \leq y$.
\end{lemma}

\begin{proof}
Suppose that $y<x$. It follows that 
\[\frac{\varphi(y)}{\phi(y)}=\frac{G(y,\theta)/ \phi(y)}{G(y,\theta) /\varphi(y)} > \frac{G(x,\theta) / \phi(x)}{G(x,\theta) / \varphi(x)}=\frac{\varphi(x)}{\phi(x)},\]
contradicting the fact that $\frac{\varphi}{\phi}$ is strictly increasing.
\end{proof}

\begin{example} \label{ex:triangle}
Let $X$ be Brownian Motion on $[0,2\pi]$ killed at $0$ and at $2\pi$. Let $c \equiv 0$ and $\Theta=\R^+$ and $G(x,\theta)=\theta {\big |}\sinh(x \sin(x)){\big |}$ and suppose $\rho=1/2$. Then $\varphi(x)=\sinh(x)$ and $\phi(x)=\sinh(2\pi-x)$. Now fix $\theta=1$ and define $\triangle_+$ and $\triangle_-$ as in Lemma \ref{l:triangle}. We calculate $\triangle_+=\{\pi/2,3\pi/2\}$ and $\triangle_- \approx \{5.14\}$. If $X_0$ lies to the left (right) of an element in $\triangle_+ $ $(\triangle_-)$ then an upper (lower) threshold is optimal. If $X_0$ lies between the largest element in $\triangle_+$ and the smallest element in $\triangle_-$ then a threshold strategy is not optimal.  
\begin{figure}[!ht]
\begin{center}
\includegraphics[height=4.5cm,width=7cm]{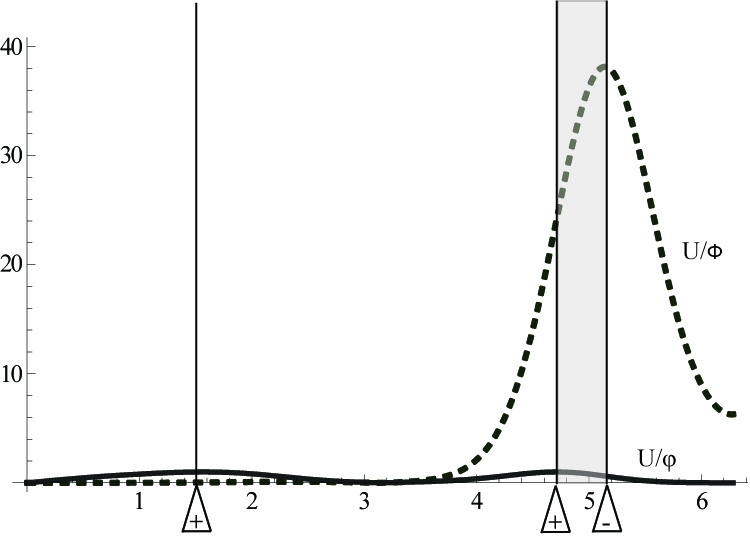}
\caption{Picture for $\theta=1$. $\frac{U}{\phi}$ is represented by the dashed line and $\triangle_-$ is a singleton. $\frac{U}{\varphi}$ is represented by the solid line and $\triangle_+$ consists of two points.
There is no optimal threshold strategy if $X_0$ lies in the shaded region.}
\end{center}
\end{figure}
\end{example}
In general, given a family of forward problems over an interval $\Theta$, we may find that threshold stopping is optimal on the whole interval $\Theta$, on a subset of $\Theta$ or nowhere on $\Theta$. We will temporarily assume that the forward problem (\ref{eq:forward}) is such that a threshold strategy is optimal on the whole parameter space.  Later, in Section \ref{s:monotone} we will see how to relax the assumption.
 
\begin{assumption} \label{ass:threshold}
For all $\theta \in \Theta$ it is optimal to stop at a threshold above $X_0$.
\end{assumption}
There is, as will always be the case, a parallel theory when the optimal thresholds are below $X_0$, compare Remark \ref{r:parallel}.

\subsection{The envelope theorem}
We will now derive our main result for the parameter dependence of the value function through an envelope theorem. The aim is to derive an expression for the derivative 
of $V$. 

For a fixed parameter $\theta$ let $\displaystyle \Xstar(\theta)=\armax_{x \in \intr(I)}\left[\frac{U(x,\theta)}{\varphi(x)}\right]$. Then $\Xstar(\theta)$ is the set of possible threshold strategies for a fixed parameter $\theta$. We will let $\Xstar(\Theta)$ denote the collection of all threshold strategies for the parameter space. Letting %$x_-=\inf\{x: x \in \Xstar(\Theta)\}$ and
$x_+=\sup\{x:x \in \Xstar(\Theta)\}$, we have that  $\Xstar(\Theta) \subseteq [X_0,x_+]$. Recall the definition of the {\it early stopping reward}. We abuse the notation slightly by setting $E(\theta)=V(X_0,\theta)-R(X_0,\theta)$, making the dependence on the starting value implicit. Let us also set $\eta(\theta)=\log(E(\theta))$. The following Proposition follows from an envelope theorem, see Corollary 4 in Segal and Milgrom, \cite{Milgrom}.

\begin{proposition} \label{p:envelope}
If $[X_0,x_+] \subseteq \intr(I)$, $U(x,\theta)$ is upper-semicontinuous in $x$ and $U_\theta(x,\theta)$ is continuous on $[X_0,x_+] \times \Theta$ then $V$ is Lipschitz continuous on $(\theta_-,\theta_+)$ and the one-sided derivatives are given by 
\begin{eqnarray*}
E'(\theta-)&=&\min_{x(\theta) \in \Xstar(\theta)} \frac{U_\theta(x(\theta),\theta)}{\varphi(x(\theta))} \\
E'(\theta+)&=&\max_{x(\theta) \in \Xstar(\theta)} \frac{U_\theta(x(\theta),\theta)}{\varphi(x(\theta))}.
\end{eqnarray*}
$E$ is differentiable at $\theta$ if and only if $\left\{\frac{U_\theta(x,\theta)}{\varphi(x)} : x(\theta) \in \Xstar(\theta)\right\}$ is a singleton. In particular we then have
\begin{equation} \label{eq:logequiv}
\frac{d}{d\theta}\eta(\theta)=u_\theta(x(\theta),\theta),
\end{equation}
for $x(\theta) \in \Xstar(\theta)$ where $u(x,\theta)=\log(U(x,\theta))$. 
\end{proposition}
\begin{remark}
Equation (\ref{eq:logequiv}) follows by combining the equations
$E'(\theta)=\frac{U_\theta(x(\theta),\theta)}{\varphi(x(\theta)}$ (a consequence of the envelope theorem in Milgrom \cite{Milgrom}) and 
$E(\theta)=\frac{U(x(\theta),\theta)}{\varphi(x(\theta))}$
(Lemma \ref{l:oneside}). 
\end{remark}
\begin{remark}
The condition $[X_0,x_+] \subseteq \intr(I)$ is satisfied if the boundary points of $I$ are accessible.
\end{remark}

\begin{corollary} \label{c:representation}
If the conditions in Proposition \ref{p:envelope} are satisfied then for any $\theta,\theta' \in \Theta$, \[E(\theta)-E(\theta')=\int_{\theta'}^\theta \frac{U_\theta(x(s),s)}{\varphi(x(s))} ds,\]
where $x(s)$ is a selection from $\Xstar(s)$.
\end{corollary}

\begin{corollary} \label{c:Gittins1}
Suppose $G(x,\theta) \equiv G(\theta)$ is continuously differentiable and $c(x,\theta)=c(x)$. If $V$ is differentiable at $\theta$ then  
\[V'(\theta)=G'(\theta) \E_{X_0}[e^{-\rho H_{x(\theta)}}],\]
for $x(\theta) \in \Xstar(\theta)$.
\end{corollary}

\begin{example}
In Example \ref{ex:fullgeo} if $\mu<\rho$,
\[
V'(\theta)= \left\{\begin{array}{ll}
\left(\frac{\theta c_- (\rho-\mu)}{c_-+1}\right)^{c_-}  &\; 0 <\theta < \frac{1+c_-}{(\rho-\mu)c_-} \\
1 &\; \theta \geq \frac{1+c_-}{(\rho-\mu)c_-} 
\end{array}\right\}.
\]
\end{example}

Parameter dependence of stopping problems is a common theme in the literature on multi-armed bandits in which a special case of the general forward problem, which we will call the {\it standard problem}, is studied. %The Gittins index from multi-armed bandit theory provides a natural intuition for results in this article, see e.g Corollary \ref{c:Gittins1}.
\begin{definition}
If $G(x,\theta)=G(\theta)$ and $c(x,\theta) \equiv c(x)$ then the problem forward problem (\ref{eq:forward}) is called the {\it standard (forward) problem}.
\end{definition}

The preceding Corollary {\ref{c:Gittins1} is the analogue in a diffusion setting of Lemma 2 in Whittle \cite{Whittle}. As in \cite{Whittle} our setup allows for points of 
non-differentiability and for the possibility of multiple optimal thresholds above the starting point. In contrast, existing results in the diffusion setting, see for instance Karatzas \cite{Karatzas} (Lemma 4.1) make strong assumptions on the diffusion and on $c$ which ensure that $\Xstar(\theta)$ is single valued and that the value function is differentiable in the parameter. 

In general, the optimal stopping thresholds for a parameter are given by a set-valued map $\Xstar:\Theta \rightarrow I$. We will now define the inverse map from the domain of the diffusion to the parameter space. %If $\Xstar$ is strictly increasing and single-valued we can consider its well-defined inverse. However, even in the more general case we will see that we can define a function
%$\thstar:\intr(I) \rightarrow \Theta$ which can be seen as the inverse map to $\Xstar$. 
\begin{definition} \label{d:gittins}
$\Theta^*(x)$, {\it the indifference map at $x$}, is the set of parameters $\theta \in \Theta$ for which it is optimal to stop immediately when $X_0=x$. 
\end{definition}

\begin{remark}
Under additional assumptions the indifference map can be represented as a monotone function, see Corollary \ref{c:gittinssub}.
\end{remark}

The indifference map $\Theta^*$ is a natural generalisation of the allocation index common in the theory of multi-armed bandits: while we make few assumptions on the reward functions, the multi-armed bandit or dynamic allocation literature is restricted to the standard problem ($c(x,\theta)=c(x)$ and $G(x,\theta)=\theta$), see for instance Gittins and Glazebrook \cite{GittinsGlazebrook}, Whittle \cite{Whittle} and for a diffusion setting closer to the setting of this article, Karatzas \cite{Karatzas} and Alvarez \cite{Alvarez}.  

%\begin{remark}
%In the two-dimensional case the dynamic allocation problem is as follows. Let $X^1$ and $X^2$ be two diffusions defined on the same probability space
%$(\Omega, \mathcal{F}, \mathbb{F}=(\mathcal{F}_t)_{t \geq 0}, \Prob)$
%and let $\mathcal{A}=\{i(t);  t \geq 0\}$ be a progressively measurable process
%with values in $\{1,2\}$ deciding which process to engage at every time $t$. Then the aim is to calculate 
%\[\sup_{\tau, A} \E_{(X^1_0,X^2_0)}[
%\end{remark}
The following example illustrates our approach to parameter dependent stopping problems and the idea of calculating critical parameter values. Although we focus on the case when the forward problem is indexed by a single parameter, the analysis of forward problems parametrised by several parameters is analogous. 
\begin{example} \label{ex:tax} {\bf A toy model for tax effects} 
Suppose $X$ is a model for the profits of a firm: $X_t=\sigma B_t+x_0$, where $B$ is a standard Brownian Motion and $\sigma>0$. In a tax-free environment a model for the value $V$ of the firm is \[V(\rho,\delta,\sigma)=\sup_{\tau} \E_{X_0}\left[\int_0^\tau e^{-\rho t} X_t dt+e^{-\rho \tau} \delta \right],\] 
where $\delta$ is the salvage value of the firm. $U(x,\delta)=\delta-x/\rho$ is decreasing in $x$ and we look for an optimal stopping threshold below $X_0$. We have $\phi(x)=e^{-\sqrt{2\rho}x/\sigma}$
and $R(x)=\frac{x}{\rho}$. Let $x^*_1$ be the optimal threshold (investment decision) in the tax-free environment above. We calculate $x^*_1=x^*_1(\rho,\delta,\sigma)=\min\{\delta \rho - \frac{\sigma}{\sqrt{2\rho}},X_0\}$. 

%Note that $x^*_1$ is increasing as a function of $\delta$ and decreasing as a function of $\sigma$. 
Now consider what happens to the value of the firm if profits are taxed. Suppose that profits are taxed at a rate $\theta$, and that the 
tax-base at time $t$ is $X_t-d$, where $d$ represents a tax-deductible depreciation expense (or some other adjustment to the tax base). The post-tax profit of the firm is $Y_t=X_t-\theta(X_t-d)$. The decreasing solution to (\ref{eq:differentialsnice}) is $\phi^Y(x)=\exp\left(\frac{-x\sqrt{2\rho}}{(1-\theta)\sigma)}\right)$ while $R^Y(x)=R((1-\theta)x+\theta d)=\frac{(1-\theta)x+\theta d}{\rho}$. The optimal threshold $(x^*_2)$ for the after-tax investment problem is 
\[x^*_2(\theta,\rho,m,\sigma,d) =\min \left\{ \frac{\rho \delta-\theta d}{1-\theta}-\frac{\sigma(1-\theta)}{\sqrt{2\rho}}, X_0 \right\}.\]

In taxation theory, a tax-rate is neutral if it does not change investment decisions. It is sometimes considered desirable for taxes to be neutral, see for instance Samuelson \cite{SamuelsonTax}. Let $\theta_N$ denote the neutral tax rate in this problem. To compute $\theta_N$ we  solve $x^*_1(\rho,\delta,\sigma)=x^*_2(\theta_N,\rho,\delta,\sigma,d)$ for $\theta_N$, to find \[1-\theta_N(\rho,\delta,\sigma,d)=\sqrt{2\rho} \  \frac{d-\delta \rho}{\sigma}.\]
Finally we check that $\theta_N \in (0,1)$ if and only if 
$0 < d-\delta p<\sigma/\sqrt{2\rho}$. Similarly, given a tax rate $\theta$ we could calculate the depreciation adjustment $d^*(\rho,\delta,\sigma,\theta)$ so that the investment decision is unchanged, which is the idea in Samuelson \cite{SamuelsonTax}. See also Klimmek \cite{RK} for analysis of the relationship between tax levels, risk preferences and decisions under uncertainty. 
\end{example}
In Example \ref{ex:tax}, the optimal thresholds are monotone in one or more of the parameters. In the next section we will derive natural conditions for the monotonicity of threshold strategies $\Xstar$. We will see that if $\Xstar$ is monotone then we can relax Assumption \ref{ass:threshold}. 

\subsection{Monotonicity of the optimal stopping threshold in the parameter value} \label{s:monotone}

We will say that $\Xstar$ is increasing (decreasing) if $x \in \Xstar(\theta)$ and $x' \in \Xstar(\theta')$ with $\theta \leq \theta'$ implies $x \leq (\geq) x'$. 

\begin{definition} $ \ $
\begin{enumerate}
\item[i)] A function $f:\R^2 \rightarrow \R$ is supermodular in $(y,z)$ if for all
$y' > y$, $f(y',z)-f(y,z)$ is increasing in $z$ and if for all $z'>z$, $f(y,z')-f(y,z)$ is increasing in $y$. 
Equivalently, $f$ is supermodular if $f(\max\{y',y\},\max\{z',z\}) + f(\min\{y',y\},\min\{z',z\}) \geq f(y,z)+f(y',z')$ for all $(y,z)$.
\item[ii)] If the inequalities in i) are strict then $f$ is called {strictly supermodular}
\item[iii)] If $-f$ is (strictly) supermodular, $f$ is called (strictly) submodular. 
\item[iv)]  $f$ is (strictly) $\log$-supermodular if $\log(f)$ is (strictly) supermodular.
\end{enumerate}
\end{definition}

\begin{remark} 
Note that if $f$ is twice differentiable then $f$ is supermodular in $(y,z)$ if and only if $f_{yz}(y,z) \geq 0$ for all $y$ and $z$. 
\end{remark} 

The next Lemma follows from a straightforward application of standard techniques in monotone comparative statics to the setting of optimal stopping, see for instance Athey \cite{Athey}.

\begin{lemma} \label{l:monotonexstar}
Suppose that $U(x,\theta)=G(x,\theta)-R(x,\theta) >0$ on $\intr(I) \times \Theta$. If $U$ is $\log$-supermodular then $\Xstar$ is increasing in $\theta$.
\end{lemma}

\begin{proof}
Suppose that $\theta > \hth$. 
$\Xstar(\theta)$ and $\Xstar(\hth)$ are non-empty by Assumption \ref{ass:threshold}. Define a function $f$ via $f(x,\theta)=u(x,\theta)-\psi(x)$, where $\psi(x)=\log(\varphi(x))$ (recall the definition of $\varphi$, (\ref{eq:eigenfunction})). Then $f$ is also supermodular. Now for any
$x(\theta) \in \Xstar(\theta)$ and $x(\hth) \in \Xstar(\hth)$ we have
\[0 \geq f(\max\{x(\theta),x(\hth)\},\theta)-f(x(\theta),\theta) 
\geq f(x(\hth),\hth)-f(\min\{x(\theta),x(\hth)\},\hth) \geq 0.\]
The first inequality follows by definition of $\Xstar(\theta)$ the second by supermodularity and the last inequality by definition of $\Xstar(\hth)$. Hence there is equality throughout and
$\max\{x(\theta),x(\hth)\} \in \Xstar(\theta)$ and 
$\min\{x(\theta),x(\hth)\} \in \Xstar(\hth)$. It follows that $\Xstar(\theta)$ is increasing in $\theta$.
\end{proof}

\begin{corollary} \label{c:monotonexstar}
If $U$ is $\log$-submodular then $\Xstar(\theta)$ is decreasing in $\theta$. 
\end{corollary}

\begin{remark} \label{r:cutout}
It may be the case that $U(x,\theta)$ takes both strictly positive and negative values on $\intr(I) \times \Theta$. In this case it is never optimal to stop at $x'$ if $U(x',\theta) \leq 0$ and so we need only check supermodularity on the set $\{(x,\theta): U(x,\theta) > 0)\}$.
\end{remark}

Monotonicity of the optimal stopping threshold will play a crucial role in our analysis of inverse optimal stopping problems because it leads to the notion of {\it indexability}. We will say that a stopping problem is {\it indexable} if $\Xstar$ is monotone. 

The notion of indexability is vital in dynamic allocation theory, leading to the natural heuristic of operating the project with the highest allocation index, i.e. `playing-the-leader'. Most recently, Glazebrook et al. \cite{Glazebrook} have generalised the notion of indexability to a large class of resource allocation problems. We note, however, that the definition of indexability presented here does not arise out of a Lagrangian relaxation of an original problem involving only a running reward as is the case in \cite{Glazebrook}. Instead, in the context of a optimal stopping problems, indexability is a natural feature in the monotone comparative statics of parametrised families of forward problems. 

The following assumption will ensure that $\Xstar$ is increasing. There will be a parallel set of results when $\Xstar$ is decreasing. 
\begin{assumption} \label{ass:increasing}
$U(x,\theta)>0$ and $\log(U)$ is supermodular.  
\end{assumption}

\begin{example} \label{ex:bessel2}
Recall Example \ref{ex:ResBessel}. Let $X$ be a three-dimensional Bessel process started at $1$, $\rho=1/2$, $c(x,\theta)=\theta \cos(x)$ and $G \equiv 0$. We have $\varphi(x)=\frac{\sinh(x)}{\sinh(1) x}$. Note that $c(x,\theta)$ is both $\log$-supermodular and $\log$-submodular. Suppose $\theta > 0$.  $\frac{-R(x,\theta)}{\varphi(x)}$ attains its maximum 
at $\hat{x}$ where $\hat{x} \approx 2$ is the smallest solution to the equation $\coth(x)(x \cos(x)-\sin(x))+x\sin(x)=0$. For $\theta<0$, the maximum is attained at the second smallest root of the same equation and $\hat{x} \approx 5.4$.
\begin{figure}[h]
\begin{center}
\includegraphics[height=4cm,width=7cm]{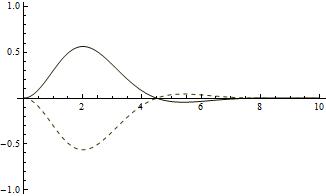}
\caption{$\frac{-R(x,\theta)}{\phi(x)}$ for $\theta=1$ (solid line) and $\theta=-1$ (dashed line).}
\end{center}
\end{figure}
Hence we find that $\Xstar(\theta)$ is decreasing. This does not contradict Lemma \ref{l:casesincreasing} because Assumption \ref{ass:increasing} is violated: The set of points where $-R(x,\theta)$ is positive when $\theta>0$ and stopping is feasible coincides with the set of points where $-R(x,\theta)$ is negative (and stopping is therefore not feasible) when $\theta<0$. Compare Remark \ref{r:cutout}. 
\end{example}

In general, if we remove Assumption \ref{ass:threshold}, a threshold strategy may never be optimal or it may only be optimal on some subset of parameters in $\Theta$. In the following we will show that if $U(x,\theta)$ is $\log$-supermodular then a threshold strategy will be optimal for all parameters in a sub-interval of $\Theta$. 

Let $\theta_R$ be the infimum of those values in $\Theta$ for which $\Xstar(\theta)=\emptyset$. If $\Xstar(\theta)=\emptyset$ for all $\theta \in \Theta$ then we set $\theta_R=\theta_-$. 
\begin{lemma} \label{l:monotonesubdiff}
The set of $\theta \in \Theta$ where $\Xstar(\theta)$ is non-empty (threshold stopping is optimal) forms an interval with end-points $\theta_-$ and $\theta_R$.
\end{lemma}
\begin{proof}
Let $\varrho$ denote the right end-point of $I$. Suppose $\Xstar(\hth) \neq \emptyset$ and $\theta \in (\theta_-,\hth)$. We claim that $\Xstar(\theta)\neq \emptyset$.

Fix $\hx \in \Xstar(\hth)$. Then $E(X_0,\hth) = 
u(\hx,\hth) - \psi(\hx)$ and
\begin{equation}
\label{eq:subdiff1}
u(\hx,\hth) - \psi(\hx) \geq  u(x,\hth) - \psi(x),
\hspace{20mm} \forall x < \varrho,
\end{equation}
and for $x = \varrho$ if $\varrho \in \intr(I)$. We write the remainder of the proof
as if we are in the case $\varrho \in \intr(I)$; the case when $\varrho \notin \intr(I)$ involves 
replacing $x \leq \varrho$ with $x < \varrho$.

Fix $\theta<\hth$. We want to show
\begin{equation}
\label{eq:subdiff2}
u(\hx,\theta) - \psi(\hx) \geq  u(x,\theta) - \psi(x),
\hspace{20mm} \forall x \in (\hx, \varrho],
\end{equation}
for then
\[ \sup_{x \leq \varrho} \{ u(x,\theta) - \psi(x) \} =  \sup_{x \leq \hx}  
\{ u(x,\theta) - \psi(x) \},
\]
and since $u(x,\theta) - \psi(x)$ is continuous in $x$ the supremum is 
attained.

Since $u$ is supermodular by assumption we have for $x \in 
(\hx,\varrho]$

\begin{equation}
\label{eq:subdiff3}
u(\hx,\hth) - u(\hx,\theta) \leq  u(x,\hth) - u(x,\theta).
%\hspace{20mm} \forall x \in (\hx, \xi),
\end{equation}
Subtracting (\ref{eq:subdiff3}) from (\ref{eq:subdiff1}) gives 
(\ref{eq:subdiff2}).
\end{proof}

In the {\it standard case}, determining whether $u(x,\theta)=\log(G(x,\theta)-R(x,\theta))$ is supermodular is simplified by the following result. 

\begin{lemma} \label{l:casesincreasing}
Suppose that the boundary points of $X$ are inaccessible. 
\begin{enumerate}
\item[] If $G \equiv 0$ then $\Xstar(\theta)$ is increasing in $\theta$ if and only if $-c(x,\theta)$ is $\log$-supermodular.
\item[] In the {\it standard case} $G(\theta)-R(x)$ is $\log$-supermodular if and only if $Q(x,\theta)$ is $\log$-supermodular where $Q:I \times \Theta \rightarrow \R$, $Q(x,\theta)=\rho G(\theta)-c(x)$.   
\end{enumerate}
\end{lemma}

\begin{proof}
Athey \cite{Athey} and Jewitt \cite{Jewitt} prove that $f: \R^2 \rightarrow \R^+$ and $h: \R^2 \rightarrow \R^+$ are $\log$-supermodular if and only if $\int_s f(x,s) h(s,\theta) ds$ is $\log$-supermodular. The first statement now follows from the fact (e.g Alvarez \cite{Alvarez} or Rogers \cite{rogers}, V.50) %or Mandl \cite{Mandl}, II.5) 
that $R(x,\theta)=\int_I r(x,y) c(y,\theta) m(dy)$, where $r(x,y)$ is a product of two single-variate functions and hence $\log$-supermodular. 

For the second statement %let $\hat{R}(x,\theta)=\E_x\left[\int_0^\infty e^{-\rho t} Q(X_t,\theta) dt \right]$. 
note that
$\E_x\left[\int_0^\infty e^{-\rho t} (\rho G(\theta)-c(X_t))dt \right]=G(\theta)-R(x)$. By the result of Athey and Jewitt, $G(\theta)-R(x)$ is $\log$-supermodular if and only if $Q(x,\theta)$ is $\log$-supermodular.
\end{proof}

\section{Inverse optimal stopping problems}
In this section our aim is to recover diffusions consistent with a given value function for a stopping problem. We recall that when $c\equiv0$ and $G(x,\theta)=(\theta-x)^+$, the problem has the interpretation of recovering price-processes consistent with perpetual American put option prices. 
Consider instead a situation in which an investor is considering whether to invest in a dividend bearing stock that can be liquidated at any time for capital gains. The capital gains depend on a parameter, e.g. a tax or subsidy rate.  In this context, the indifference index has the natural interpretation of the parameter level(s) at which the investor is indifferent about the stock: at the critical level, the optimal policy would be to sell the stock immediately hence there is no expected gain from investment. The question that we ask in this section is whether we can recover an investor's model for the asset price process given his valuation and investment indifference levels for the parameter.

The problem of recovering investor preferences from given information is a natural problem in economics and finance. % Rather than specifying overarching models for financial or economic processes and then solving optimisation problems, the idea underlying work on inverse problems is to infer information about models  given empirical information.
In economics, the question of recovering information about an agent's preferences given their behaviour dates back to Samuelson's work \cite{SamuelsonPref} on revealed preferences. Work on inverse investment problems includes Black \cite{BlackInverse}, Cox and Leland \cite{CoxLeland}, He and Huang \cite{HeHuang} and most recently Cox et al. \cite{CoxInverse}. Rather than calculating an optimal consumption/portfolio policy for a given agent (with a given utility function), the literature in this area aims at recovering utility functions consistent with given consumption and portfolio choices. These `inverse Merton problems' have three fundamental aspects. The first is the specification of a model for an agents' wealth process. The dynamics of the model are given by a dynamic budget constraint and are fully determined given an agent's consumption and investment policy. The second fundamental aspect is an assumption on how the agent values the wealth developing out of his investment and consumption activity. It is assumed that he maximises utility. The third aspect, which is the crux of inverse Merton problems, is to determine the agents' utility, or gain functions given the wealth dynamics and assuming the agent is utility maximising.

As in the inverse Merton problem, three fundamental quantities emerge in the study of the inverse perpetual optimal stopping problems considered here. 1) A model for the underlying random process,
2) valuation of the investment and 3) investor preferences (indifference levels). As we will see below, if we are given only an agent's valuation then the inverse problem is ill-posed. There will then in general be infinitely many models solving the inverse problem, each corresponding to a choice of indifference index. As in the inverse Merton problems where both a model for the wealth process and the assumption of valuation via utility maximisation are given, we require two of the three pieces of information inherent to the perpetual-horizon investment problem to construct a solution: given a value function and (admissible) investor indifference levels, a consistent model is uniquely specified on the domain of the indifference index. 

\subsection{Setup}
As before, let $\Theta$ be an interval with end-points $\theta_-$ and $\theta_+$. Let us assume that we are given $V=\{V(\theta) \ ; \ \theta \in \Theta\}$, $G=\{G(x,\theta) \ ;\  x \in \R, \ \theta \in \Theta\}$, $c=\{c(x) \ ; \ x \in \R\}$ and $X_0$. \newline
$ \ $
\newline
{\bf Inverse Problem}: {\it Find a generalised diffusion $X$ such that $V_X=V$ is consistent with one-sided stopping above $X_0$.}

\begin{remark}
In developing the theory we focus on threshold strategies above $X_0$. There is, as ever, an analogous inverse problem for threshold strategies below $X_0$. See for instance Example \ref{ex:inversegeo1}.
\end{remark}

We will make the following regularity assumption. 

\begin{assumption} \label{ass:inverse}
$V: \Theta \rightarrow \R$ is differentiable and $(x,\theta) \rightarrow G(x,\theta)$ is twice continuously differentiable.
\end{assumption}

As in Hobson and Klimmek \cite{HobsonKlimmek:10}, we will use generalized convex analysis of the $\log$-transformed stopping problem to solve the inverse problem.  

\subsection{u-convex Analysis}
\begin{definition}
Let $A, B$ be subsets of $\R$ and let $u:A \times B \rightarrow \R$ be a bivariate function. The $u$-convex dual of a function $f: A \rightarrow B$ is denoted $f^u$ and defined to be $f^u(z)=\sup_{y \in A} [u(y,z)-f(y)]$.
\end{definition}

\begin{definition}
A function $f$ is $u$-convex if $f^{uu}=(f^{u})^{u}=f$.
\end{definition}

\begin{definition}\label{def:subdifferential2}
The $u$-subdifferential of $f$ at $y$ is defined by
\[\partial^u f(y) = \{ z \in B : f(y) + f^u(z) = u(y,z) \}, \]
or equivalently
\[\partial^u f(y) = \{ z \in B : u(y,z)- f(y) \geq u(\hat{y},z) - 
f(\hat{y}), \forall \hat{y} \in A \} . \]
\end{definition}

If $S$ is a subset of $A$ then we define $\partial^u f(S)$ to be 
the union of $u$-subdifferentials of $f$ over all points in $S$.

\begin{definition}
$f$ is $u$-subdifferentiable at $y$ if $\partial^u f(y) \neq 
\emptyset$. $f$ is $u$-subdifferentiable on $S$ if it is 
$u$-subdifferentiable for all $y \in S$, and $f$ is $u$-subdifferentiable if it is $u$-subdifferentiable on $A$.
\end{definition}

The following envelope-theorem from $u$-convex analysis will be fundamental in 
establishing duality between the value function and the $\log$-transformed eigenfunctions of consistent diffusions. The idea, which goes back to R\"{u}schendorf
\cite{ruschfrech} (Equation 73), is to match 
the gradients of $u(y,z)$ and $u$-convex functions $f(y)$, whenever
$z \in \partial^u f(y)$. The approach was also developed in Gangbo and McCann \cite{mccann} and for applications in Economics by Carlier \cite{carlier}. We refer to 
\cite{carlier} for a proof of the following result.

\begin{proposition} \label{p:uenvelope}
Suppose that $u$ is strictly supermodular and twice continuously differentiable.

If $f$ is a.e differentiable and 
$u$-subdifferentiable. Then there exists a map $\zstar: D_y 
\rightarrow D_z$ such that if $f$ is differentiable at $y$ then
$f(y)=u(y,\zstar(y))-f^u(\zstar(y))$ and
\begin{equation} \label{eq:diffsubdiff2}
f '(y) = u_y(y,\zstar(y)).
\end{equation}

Moreover, $\zstar$ is such that $\zstar(y)$ is non-decreasing.

Conversely, suppose that $f$ is a.e 
differentiable and equal to the integral of its derivative. If 
(\ref{eq:diffsubdiff2}) holds for a non-decreasing function
$\zstar(y)$, then $f$ is $u$-convex and 
$u$-subdifferentiable with $f(y)=u(y,\zstar(y))-f^u(\zstar(y))$.
\end{proposition}

\begin{remark}
If $u$ is strictly $\log$-submodular the conclusion of
Proposition \ref{p:uenvelope} remains
true, except that $z^*(y)$ and $y^*(z)$ are non-increasing.
\end{remark}
The subdifferential $\partial^u f$ may be an interval in which case $z^*(y)$ may be taken to be any element in that interval. By Lemma \ref{l:monotonexstar}, $z^*(y)$ is non-decreasing. We observe that since $u(y,\zstar(y))=f(y)+f^u(\zstar(y))$ we have $u(y^*(z),z)=f(y^*(z))+f(z)$ and $y^*(z) \in \partial^u f^u(z)$ so 
that $y^*$ may be defined directly as an element of $\partial^u f^u$.
If $\zstar$ is strictly increasing then $y^*$ is just the inverse of $\zstar$.

\subsection{Application of $u$-convex Analysis to the Inverse Problem}
We introduce the notation that we will use for our inverse stopping problem framework. The main change over the previous section is that we will highlight dependence on the (unknown) speed measure $m$ and scale function $s$.

We wish to recover a speed measure $m$ and scale function $s$ to construct a diffusion $X^{m,s}=(X^{m,s}_t)_{t \geq 0}$, supported on a domain $I^m \subseteq \R$ such that $V_{X^{m,s}}=V$. Our approach to solving this problem is to recover solutions $\varphi_{m,s}$ and $R_{m,s}$ to the differential equations (\ref{eq:differentialsc}) and (\ref{eq:runningr}) from $V$ and to solve the two equations `in reverse' to recover the speed measure and the scale function. Let $\psi_{m,s}=\log(\varphi_{m,s})$, where $\varphi_{m,s}$ is the increasing solution to (\ref{eq:differentialsc}) with $\varphi_{m,s}(X_0)=1$ and let $R_{m,s}(x)=\E_{x}[\int_0^\infty e^{-\rho t} c(X^{m,s}_t) dt]$. We will say that functions $R_{m,s}$, $\varphi_{m,s}$, $\Xstar$, etc. are {\it consistent} with the inverse problem if there exists a diffusion  $X^{m,s}$ such that $V_{X^{m,s}}=V$. Our approach involves establishing $\psi_{m,s}$ and $\eta_{m,s}(\theta)=\log(V(\theta)-R_{m,s}(X_0))$ as $u_{m,s}$-convex dual functions, where $u_{m,s}(x,\
theta)=\log(G(x,\theta)-R_{m,s}(x))$. Denote the $u_{m,s}$-convex duals of $\varphi_{m,s}$ and $\eta_{m,s}$ by $\varphi_{m,s}^u$ and $\eta_{m,s}^u$ respectively.

The solution to the inverse problem hinges on Proposition \ref{p:uenvelope}.  
Let us briefly highlight the connection between Proposition \ref{p:uenvelope} and Proposition \ref{p:envelope}. Suppose $V=V_{X^{m,s}}$ and that $u_{m,s}(x,\theta)$ is strictly $\log$-supermodular and twice continuously differentiable. Then $\eta_{m,s}(\theta)$ is $u_{m,s}$-convex with $u_{m,s}$-subdifferential $\Xstar$, i.e. $\eta_{m,s}(\theta)=\sup_{x \in int(I^m)}[u_{m,s}(x,\theta)-\psi_{m,s}(x)]=u_{m,s}(\xstar(\theta),\theta)-\psi_{m,s}(\xstar(\theta))$, for some optimal stopping threshold $\xstar(\theta) \in \Xstar(\theta)$. Hence by Proposition \ref{p:uenvelope}, we have \[\frac{V'(\theta)}{V(\theta)-R_{m,s}(X_0)}=\frac{G_\theta(\xstar(\theta))}{G(\xstar(\theta),\theta)-R_{m,s}(\xstar(\theta))}.\] Substituting for $\varphi_{m,s}(\xstar(\theta))$ we find
\[V'(\theta)=\frac{G_\theta(\xstar(\theta),\theta)}{\varphi_{m,s}(\xstar(\theta))},\]
which is the expression in Proposition \ref{p:envelope} when $V$ is differentiable.

In the following, whenever $u_{m,s}$ is strictly supermodular and $V_{X^{m,s}}=V$ we will let $\xstar$ denote the non-decreasing function satisfying $\eta_{m,s}'(\theta)=\frac{\partial}{\partial \theta} u_{m,s}(\xstar(\theta),\theta)$. Then $\Xstar$ is the set of points on the graph of $\xstar$. We will call $\thstar=\xstar^{-1}$ the {\it indifference index}.

\begin{corollary} \label{c:gittinssub}
Suppose $u_{m,s}$ is strictly supermodular and twice continuously differentiable. If $\psi_{m,s}$ is $u_{m,s}$-subdifferentiable then the indifference index at $x \in int(I^m)$ satisfies $\psi_{m,s}'(x)=\frac{\partial}{\partial x}u_{m,s}(x,\thstar(x))$. Moreover, $\thstar$ is non-decreasing. 
\end{corollary}
% Finally, define {$\displaystyle \psi^u_{m,s,X_0}(\theta)=\sup_{x \geq X_0, x \in int(I^m)}[u_{m,s}(x,\theta)-\psi_{m,s}(x)]$}. 
\subsection{Recovering consistent diffusions}
The following theorem provides necessary and sufficient conditions for a diffusion $X^{m,s}$ to be the solution to the inverse stopping problem. 

\begin{proposition} \label{p:consistency}
$X^{m,s}$ solves the inverse problem if and only if $\varphi_{m,s}$ and $R_{m,s}$ satisfy the following two conditions.
\begin{enumerate}
\item[i)] For all $\theta \in \Theta$, $\displaystyle \psi_{m,s}^u(\theta)=\sup_{x \in int(I^m), \ x \geq X_0}[u_{m,s}(x,\theta)-\psi_{m,s}(x)]$,
\item[ii)] $\psi^u_{m,s}(\theta) =\log(V(\theta)-R_{m,s}(X_0))$.
\end{enumerate}
\end{proposition}
\begin{proof}
If $X^{m,s}$ is consistent with $V$, $G$, $X_0$ and one-sided stopping above $X_0$ then  
\begin{eqnarray*}
\log(V(\theta)-R_{m,s}(X_0))&=&\log(V_{X^{m,s}}(\theta)-R(X^{m,s}_0)) \\
&=& \sup_{x \in int(I^m), \ x \geq X_0}[u_{m,s}(x,\theta)-\psi_{m,s}(x)] \\
&=& \sup_{x \in int(I^m)}[u_{m,s}(x,\theta)-\psi_{m,s}(x)] \\ 
&=& \psi^u_{m,s}(\theta). 
\end{eqnarray*}
On the other hand, if the two conditions are satisfied then we can construct a diffusion $X^{m,s}$ with starting point $X_0$. The first condition implies that one-sided stopping above $X_0$ is optimal while the second condition ensures that $V_{X^{m,s}}=V$.
\end{proof}

It is intuitively clear that a value function contains information about the dynamics of a consistent diffusions above the starting point. If $x \geq X_0$, and the indifference index $\thstar$ is known then the solution to (\ref{eq:differentialsc}) must
satisfy $\varphi(x)=\frac{U(x,\thstar(x))}{V(\thstar(x))-R(X_0)}$. Thus if we can calculate $R(X_0)$ and the indifference index for all $x \geq X_0$ then we can calculate $\varphi$ above the starting point. We can then recover pairs of scale functions and speed measures consistent with the solution above $X_0$ through (\ref{eq:differentialsc}). 

On the other hand, for $x<X_0$, the only information we have is that $\frac{U(x,\thstar(x))}{\varphi(x)}$ does not attain a maximum below the starting point, otherwise $V$ would not be consistent with one-sided stopping above $X_0$. Thus, while we may attempt to specify (unique) diffusion dynamics above $X_0$, we expect there to be a variety of consistent specifications of the diffusion dynamics below the starting point. This is analogous to the situation in Hobson and Ekstr\"{o}m \cite{HobsonEkstroem} where a unique consistent volatility co-efficient is derived below the starting point but there is freedom of choice above the starting point. The situation is similar in Alfonsi and Jourdain \cite{alfonsi2}, where information about the underlying diffusion co-efficient can only be recovered either above or below the starting point, depending on whether perpetual American call or put option prices are given.
%To illustrate these observations together with the results of Proposition \ref{p:consistency} and Proposition \ref{p:uenvelope}, let us consider two simplified inverse problems. 

The following two examples illustrate the ideas involved in recovering a consistent diffusion in the simplified setting when consistent diffusions are assumed to be either martingales (Example \ref{ex:inversegeo1}) or in natural scale and with additional information about the early stopping reward (Example \ref{ex:additionalinfo}).
\begin{example} \label{ex:inversegeo1}
Let $\Theta=\left(0,\frac{k+1}{k}\right]$ for some positive constant $k$.
Suppose $V(\theta)=(\frac{k\theta}{k+1})^k\frac{\theta}{k+1}+1$, $G(x,\theta)=\theta$, $c(x)=\rho x$ and $X_0=1$. Suppose the inverse problem is restricted to the class of diffusions that are also martingales. Then $s(x)=x$ and $R_{m,s}(x)=x$. We have $u_{m,s}(x,\theta)=\log(\theta-x)$ and calculate 
$\eta_{m,s}^u(x)=\sup_{\theta}[\log(\theta-x)-\log(V(\theta)-1)]=\log(x^{-k})$, where the maximum is attained at $\thstar(x)=\frac{x(k+1)}{k}$. To recover a consistent martingale diffusion on $\R^+$, let us extend the parameter space to $\bar{\Theta}=(0,\infty)$ and set $\thstar(x)=\frac{x(k+1)}{k}$ on $(0,\infty)$. Then we find that $\phi_{m,s}(x)=x^{-k}$ on $(0,\infty)$ is a consistent eigenfunction. It follows
that \[dX_t=\sigma X_t dB_t, \ \ X_0=1\] is consistent with $V$ where $\sigma$
satisfies $\sigma^2=\frac{2\rho}{(k+1/2)^2-1/4}$.
\end{example}

\begin{example} \label{ex:additionalinfo}
Let $\Theta=[1,\infty)$. Recall the decomposition of the forward problem by the strong Markov property (\ref{eq:strongmarkov}). Suppose we are given the optimal early stopping reward
$E(\theta)=e^{\theta^2/2}$ and the early stopping reward function
$U(x,\theta)=e^{\theta x}$ and that $X_0=0$. 
%Suppose that we are given the optimal early stopping reward  
%$U=\{U(x,\theta)=G(x,\theta)-R_{m,s}(x) ; x \in \R,  \theta \in \Theta\}$ and $E(\theta)=V(\theta)-R_{m,s}(X_0)$ and $s(x)=x$ (but not $G$, $R_{m,s}$ or $c$) Suppose $\Theta=[1,\infty)$, $E(\theta)=e^{\theta^2/2}$, $U(x,\theta)=e^{\theta x}$ and $X_0=0$. 
In this example, $\eta(\theta)=\log(E(\theta))$ is known, so we suppress the subscripts $m$ and $s$. We calculate $\sup_\theta[u(x,\theta)-\eta(\theta)]=x^2/2$ where the maximum is attained at $\thstar(x)=x$. Let us suppose that $s(x)=x$ and aim at recovering a (local)-martingale diffusion. On $\Xstar(\Theta)=[1,\infty)$,  the candidate eigenfunction for the diffusion is $\varphi_{m,s}(x)=e^{\eta^u(x)}=e^{x^2/2}$. Solving for $\sigma$ in (\ref{eq:differentialsnice}) we obtain $\sigma(x)=\frac{2 \rho}{1+x^2}$ for $x \in [1,\infty)$. We can now specify a consistent diffusion by extrapolating the indifference index. Let $\bar{\Theta}=(0,\infty)$ and set $\thstar(x)=x$ for $0 \leq x \leq 1$. By Proposition \ref{p:uenvelope}, $\psi_{m,s}$ is $u$-convex on $\Xstar(\bar{\Theta})$ if $\frac{\varphi_{m,s}'(x)}{\varphi_{m,s}(x)}=\thstar(x)=x$. Thus by setting 
$\varphi_{m,s}(x)=e^{x^2/2}$ for $x \in \R^+$ we find that the diffusion with 
dynamics \[dX_t=\sigma(x) dB_t + dL_t, \ \ X_0=0, \ \   \sigma^2(x)=\frac{2 \rho}{1+x^2},\] where $L$ is the local time at $0$, is consistent with $V$.

In general we can choose any increasing function $\thstar$ on $[0,1)$ with $\thstar(1-)=1$ as long as the recovered function $\varphi_{m,s}$ is an eigenfunction for a consistent diffusion. For instance, the choice $\thstar(x)=\frac{3}{4} x^{1/2}$ for $0 \leq x<1$ leads to $\varphi_{m,s}(x)=\exp\left(\frac{x^{3/2}}{2}\right)$ for $0 \leq x < 1$. For this choice of extension and again setting $s(x)=x$, the consistent diffusion co-efficient is
\[\sigma^2(x)=\left\{\begin{array}{ll}
\frac{32 \rho x^{1/2}}{6+9x^{3/2}}  &\; 0 \leq x < 1 \\
\frac{2 \rho}{1+x^2} &\; x \geq 1  
\end{array}\right\}
.\]
Note that for this extension $\varphi'_{m,s}$ jumps at $1$ and since
\[\varphi'_{m,s}(1+)-\varphi'_{m,s}(1-)=2 \rho \varphi_{m,s}(1) m(\{1\}),\]
we have $m(\{1\})=\frac{1}{8\rho}$ (compare Example 1.4.3). Hence the increasing additive functional $\Gamma_u$ includes a multiple of the local time at $1$ and the diffusion $X^m$ is `sticky' at $1$.
\end{example}
%The previous example illustrates the steps involved in solving the inverse problem. The first step is to recover diffusion co-efficients on $\Xstar(\Theta)$. The second step is to recover a $u_{m,s}$-convex extension to the recovered eigenfunction $\varphi_{m,s}$ on a full domain. If $u_{m,s}$ is $\log$-supermodular then by Proposition \ref{p:uenvelope} we can define an extension by specifying a consistent Gittins index. 
For the general case, the main difficulty over the previous simplified examples of inverse problems is having to recover both a speed measure and a non-trivial scale function. This means that we must recover $R_{m,s}$ as well as $\varphi_{m,s}$ to obtain two equations (\ref{eq:differentialsc}), (\ref{eq:runningr}) for the two unknown quantities. 

\subsection{Recovering diffusions through a consistent indifference index}
Suppose that an indifference index $\thstar:I^m \rightarrow \Theta$ is consistent with a solution to the inverse problem, $X^{m,s}$. Then $\xstar=\thstar^{-1}$ is the optimal threshold strategy for the corresponding forward problem and by Proposition \ref{p:envelope}, 
\begin{equation} \label{eq:inversevarphi}
V'(\thstar(x))=\frac{G_\theta(x,\thstar(x))}{\varphi_{m,s}(x)}.
\end{equation}

\begin{lemma} \label{l:inverse-Rrep}
If $\thstar$ is consistent with the inverse problem then all consistent diffusions $X^{m,s}$ satisfy
\[\E_x\left[\int_0^\infty e^{-\rho t} c(X^{m,s}_t) dt\right]=G(x,\thstar(x))+\varphi_{m,s}(x)(R_{m,s}(X_0)-V(\thstar(x)))\]
for all $x \in int(I^m)$.
\end{lemma}

\begin{proof}
By the definition of $\thstar$, $x \in \Xstar(\thstar(x))$. It follows by (\ref{eq:logequiv}) that
\[\frac{V'(\thstar(x))}{V(\thstar(x))-R_{m,s}(X_0)}=\frac{G_\theta(x,\thstar(x))}{G(x,\thstar(x))-R_{m,s}(X_0)}\] for all $x \in \Xstar(\Theta)$. Combining this equation with (\ref{eq:inversevarphi}) we have
\begin{equation} \label{eq:inverseR}
R_{m,s}(x)=G(x,\thstar(x))+\varphi_{m,s}(x)(R_{m,s}(X_0)-V(\thstar(x))).
\end{equation}
\end{proof}

Let $\hat{R}_{m,s}(x)$ be the function on $\Xstar(\Theta)$ defined $\hat{R}_{m,s}(x)=G(x,\thstar(x))-\varphi_{m,s}(x) V(\thstar(x))$. 

\begin{lemma} \label{l:hatsolves}
If $R_{m,s}(x)$ given by equation (\ref{eq:inverseR}) solves  (\ref{eq:runningr}) then so does $\hat{R}_{m,s}$.
\end{lemma}
\begin{proof}
Follows from the fact that $\varphi_{m,s}$ is a solution to the homogeneous equation (\ref{eq:differentialsc}).
\end{proof}

Given an inverse problem there will in general be many speed measures and scale functions satisfying the conditions in Proposition \ref{p:consistency}. Each solution corresponds to an optimal threshold strategy $\Xstar$. %By Proposition \ref{p:uenvelope} the solutions for which $u_{m,s}$ is strictly $\log$-supermodular correspond to a monotone Gittins index $\thstar: I^m \rightarrow \bar{\Theta}$ with $\Theta \in \bar{\Theta}$. 
By definition, choosing a consistent indifference index is equivalent to choosing a consistent threshold strategy. Thus, rather than searching over all solutions $X^{m,s}$ satisfying the conditions in Proposition \ref{p:consistency}, we can solve inverse problems by specifying a candidate indifference index.
%hich will determine a consistent diffusion $X^{\hat{m},\hat{s}}$ on $\Xstar(\Theta)$. We extend the consistent diffusion to define a candidate solution $X^{m,s}$ of the inverse problem on a domain $I^m \supseteq \Xstar(\Theta)$, where $\hat{m}(dx)=m(dx)$ and $\hat{s}(x)=s(x)$ on $\Xstar(\Theta)$. 
The following verification result provides a set of easily verifiable conditions for $X^{m,s}$ to solve the inverse problem.
\begin{proposition} \label{p:candidate}
$X^{m,s}$ is a solution to an inverse problem if the following conditions are satisfied.
\begin{enumerate}
\item[i)] $u_{m,s}$ is strictly supermodular and twice continuously differentiable and $\varphi_{m,s}$ is differentiable almost everywhere,
\item[ii)] there exists a monotone function $\xstar: \bar{\Theta} \rightarrow I^m$ with inverse $\thstar$ such that $\Theta \subseteq \bar{\Theta}$, $\xstar(\theta) \geq X_0$ and such that whenever $\psi_{m,s}$ is differentiable \[\psi'_{m,s}(x)= \frac{\partial}{\partial x} u_{m,s}(x,\thstar(x)),\] 
\item[iii)] $\eta_{m,s}=\psi_{m,s}^u$.
\end{enumerate}
\end{proposition}

\begin{proof}
By Proposition \ref{p:uenvelope} and conditions i) and ii)
$\psi_{m,s}$ is $u_{m,s}$-convex. It follows from iii) and Proposition \ref{p:consistency} that $V_{X^{m,s}}=V$. 
\end{proof}

\begin{theorem}
A consistent indifference index determines a unique solution to the inverse problem on $\Xstar(\Theta)$.
\end{theorem}

\begin{proof}
Suppose $\thstar$ is a consistent indifference index. Define
$f: \Xstar(\Theta) \rightarrow \R$, $f(x)=\frac{G_\theta(x,\thstar(x))}{V'(\thstar(x))}$ and $g: \Xstar(\Theta) \rightarrow \R$, $g(x)=G(x,\thstar(x))-f(x)V(\thstar(x)))$. By (\ref{eq:inversevarphi}), $f=\varphi_{m,s}$ on $\Xstar(\Theta)$ for a consistent diffusion $X^{m,s}$, hence $f$ is a solution to (\ref{eq:differentialsc}) on $\Xstar(\theta)$. Similarly by Lemma \ref{l:hatsolves}, $g$ is a solution to 
(\ref{eq:runningr}) on $\Xstar(\Theta)$. Solving the two equations for $m$ and $s$ we recover the (unique) dynamics of a consistent diffusion on $\Xstar(\Theta)$.
\end{proof}

\begin{example} \label{ex:inversegeo}
Suppose $\Theta=\left(0,\frac{k+1}{k}\right]$ and for a positive constant $k$, $V(\theta)=\left(\frac{k \theta}{k+1}\right)^k \frac{\theta}{k+1}+1$, $X_0=1$, $G(x,\theta)=\theta$ and $c(x)=\gamma x$ where $\gamma$ is another positive constant. We define a family of indifference indices on $(0,1]$ parametrised by $\alpha >0$ via $\theta^*_\alpha(x)=\frac{x^\alpha(k+1)}{k}$. 

We will calculate candidate diffusions using Proposition \ref{p:candidate}. By (\ref{eq:inversevarphi}) we have that for $x \in X^*_\alpha(\Theta)=(0,1]$  a candidate solution to (\ref{eq:differentialsnice}) corresponding to the indifference index $\theta^*_\alpha$ is $\phi^\alpha(x)=x^{-\alpha k}$. Similarly by (\ref{eq:inverseR}) we have $R^\alpha(x)=x^\alpha+\phi^\alpha(x)(R^\alpha_{m,s}(1)-1)$ and so 
$\hat{R}^\alpha_{m,s}(x)=x^\alpha$ is a candidate solution to 
(\ref{eq:runningr}). Then, by equations (\ref{eq:differentialsnice}) and (\ref{eq:runningr}), the corresponding consistent diffusion co-efficients on 
$X^*_\alpha(\Theta)=(0,1]$ are 
\begin{eqnarray*}
\sigma_\alpha^2(x) &=& \frac{2(\rho(1+k) x^2-k\gamma x^{3-\alpha})}{k(1+k) \alpha^2} \\
\mu_\alpha(x)&=&(1+\alpha k) \frac{\rho x(1+k)-k \gamma x^{2-\alpha}}{k(1+k) \alpha^2}-\frac{\rho x}{\alpha k}. \\
 \end{eqnarray*}
Note that $\sigma_\alpha^2(x) \geq 0$ on $(0,1]$ if and only if $x^{1-\alpha} \leq \frac{\rho(1+k)}{k\gamma}$ and hence for a consistent diffusion to exist on $(0,1]$ the problem parameters must satisfy $\alpha \leq 1$ and $\rho+k(\rho-\gamma) \geq 0$. 

To specify a diffusion on $(0,\infty)$ consistent with a given $\alpha \leq 1$ on $(0,1]$, we let
\[\thstar(x)= \left\{\begin{array}{ll}
\theta^*_\alpha(x)  &\; x \in \Xstar(\Theta)=(0,1] \\
\theta^*_3(x) &\; x > 1  
\end{array}\right\}.
\]

The corresponding diffusion $X^{\alpha}$ is given by
\[dX^\alpha_t=\sigma_\alpha(X_t)dB_t+\mu_\alpha(X_t)dt, \ \ X_0=1\]
where 
\[\sigma_\alpha^2(x)=\left\{\begin{array}{ll}
\frac{2(\rho(1+k) x^2-k\gamma x^{3-\alpha})}{k(1+k) \alpha^2}  &\; 0 < x \leq 1 \\
\frac{2(\rho(1+k) x^2-k\gamma)}{9k(1+k)} &\; x > 1  
\end{array}\right\}
\]
and 
\[\mu_\alpha(x)=\left\{\begin{array}{ll}
(1+\alpha k) \frac{\rho x(1+k)-k \gamma x^{2-\alpha}}{k(1+k) \alpha^2}-\frac{\rho x}{\alpha k}  &\; 0< x \leq 1 \\
(1+3k) \frac{\rho x(1+k)-k \gamma x^{-1}}{9k(1+k)}-\frac{\rho x}{3k} &\; x > 1  
\end{array}\right\}
.\]
The particular choice of $\thstar$ on $(1,\infty)$ is convenient because it ensures that the trivial condition $\sigma_\alpha^2(x) \geq 0$ is satisfied for any choice of $\rho$, $k$ and $\gamma$ satisfying $\rho+k(\rho-\gamma) \geq 0$. 

Since both boundary points are inaccessible we have 
$R^\alpha=\hat{R}^\alpha$ or equivalently $R^\alpha(1)=1$.

Note that if we set $\gamma=\rho$ and $\alpha=1$ then we recover Example \ref{ex:inversegeo1}.
\end{example}

\begin{example} \label{ex:inverse2}
Suppose $\Theta=(1,3)$, $V(\theta)=1+\frac{\theta}{3}(\sqrt{3\theta}-\sqrt{\theta/3})$, $G(x,\theta)=\theta x$, $c(x)=1/x$ and $X_0=1$. Furthermore suppose we are given $\thstar(x)=3/x^2$ for $x \in (0,\infty)$. Then $\xstar(\theta)=\sqrt{3/\theta}$ and $\Xstar(\Theta)=(1,\sqrt{3})$. By (\ref{eq:inversevarphi}) we have $\varphi_{m,s}(x)=x^2$ and 
$R_{m,s}(x)=1/x+x^2(R_{m,s}(1)-1)$ so that $\hat{R}_{m,s}(x)=1/x$ is a candidate solution to (\ref{eq:inverseR}). The differential equations (\ref{eq:differentialsc}) and (\ref{eq:runningr}) lead to the following simultaneous equations.
\begin{eqnarray*}
\sigma^2(x)+2 \mu(x)x &=& \rho x^2, \\
\sigma^2(x)-x \mu(x) &=& \rho x^2 - x^2. 
 \end{eqnarray*}
We calculate $\sigma^2(x)=x \mu(x)+\rho x^2-x^2=\rho x^2-2\mu(x)x$ so that 
$\mu(x)=x/3$ and $\sigma^2(x)=x^2(\rho-2/3)$. It follows that we must have $\rho > 2/3$ for a solution to the inverse problem to exist. Provided this condition is satisfied, a solution to the inverse problem is 
\[dX_t=\sqrt{\rho - \frac{2}{3}} X_t dB_t + \frac{X_t}{3}dt, \ \ X_0=1.\] 
Note that the solution to the inverse problem is uniquely specified on $(1,\sqrt{3})$.
\end{example}

\section{Concluding remarks}
The main contribution of this article has been to provide a natural extension of the allocation (Gittins) index based on its role in solving forward and inverse stopping problems. In the context of the forward problem we showed that the idea of an allocation index can be extended naturally from the `standard case' to a general class of optimal stopping problems and that there are natural conditions under which the index is monotone. Indeed, we found that this extension, which we have called the indifference index, is a natural feature of the monotone comparative statics of a family of forward problems.

We showed that the indifference index parametrises solutions to inverse stopping problems. When an investment can be modelled as a perpetual horizon stopping problem, inverse stopping problems can be interpreted as inverse investment problems. In contrast to inverse Merton problems where we are given 1) a wealth process and 2) assume that the problem value arises from utility maximization in order to 3) compute consistent utility (gain) functions, in perpetual horizon inverse investment problems we are again given information about the problem value but are interested in recovering an investor's model for the underlying risky asset rather than the gain functions which are now given. In this context, the index has two natural economic interpretations. For the owner of an investment, the indifference index represents investor preferences with respect to liquidating for a terminal reward or remaining invested for a running reward and the option to liquidate later. For the potential investor, the index represents the parameter levels at which he is indifferent to the investment opportunity.

{\small \bibliography{biblio}}

\begin{thebibliography}{10}

\bibitem{alfonsi2}
A.~Alfonsi and B.~Jourdain.
\newblock General duality for perpetual {A}merican options.
\newblock {\em International Journal of Theoretical and Applied Finance},
  11(6):545--566, 2008.

\bibitem{Alvarez}
L.H.R Alvarez.
\newblock Reward functionals, salvage values, and optimal stopping.
\newblock {\em Mathematical {M}ethods of {O}perations {R}esearch},
  54(2):315--337, 2001.

\bibitem{Athey}
S.~Athey.
\newblock Comparative statics under uncertainty: {S}ingle crossing properties
  and {L}og-{S}upermodularity.
\newblock {\em Working paper, {D}epartment of {E}conomics, {MIT}}, 1996.

\bibitem{Bank}
P.~Bank and C.~Baumgarten.
\newblock Parameter-dependent optimal stopping problems for one-dimensional
  diffusions.
\newblock {\em Electronic {J}ournal of {P}robability}, 15:1971--1993, 2010.

\bibitem{Lions}
A.~Bensoussan and J.L. Lions.
\newblock {\em Applications des In\'{e}qualites Variationelles en Contr\^{o}le
  Stochastique}.
\newblock Dunod, 1978.

\bibitem{BlackInverse}
F.~Black.
\newblock Individual investment and consumption under uncertainty.
\newblock {\em Portfolio Insurance: A guide to dynamic hedging. Wiley}, pages
  207--225, 1988.

\bibitem{borodin}
A.N. Borodin and P.~Salminen.
\newblock {\em Handbook of Brownian Motion - Facts and Formulae}.
\newblock Birkh{\"{a}}user, 2nd edition, 2002.

\bibitem{carlier}
G.~Carlier.
\newblock Duality and existence for a class of mass transportation problems and
  economic applications.
\newblock {\em Adv. in Mathematical economy}, 5:1--22, 2003.

\bibitem{CoxInverse}
A.M.G. Cox, D.~Hobson, and J.~Obloj.
\newblock Utility theory front to back-inferring utility from agents' choices.
\newblock {\em arXiv preprint arXiv:1101.3572}, 2011.

\bibitem{CoxLeland}
J.C. Cox and H.E. Leland.
\newblock On dynamic investment strategies.
\newblock {\em Journal of Economic Dynamics and Control}, 24(11):1859--1880,
  2000.

\bibitem{DayanikKaratzas}
S.~Dayanik and I.~Karatzas.
\newblock On the optimal stopping problem for one-dimensional diffusions.
\newblock {\em Stochastic {P}rocesses and their {A}pplications},
  107(2):173--212, 2003.

\bibitem{HobsonEkstroem}
E.~Ekstr{\"{o}}m and D.~Hobson.
\newblock Recovering a time-homogeneous stock price process from perpetual
  option prices.
\newblock {\em Ann. Appl. Probab.}, 21(3):1102--1135, 2011.

\bibitem{mccann}
W.~Gangbo and R.J McCann.
\newblock The geometry of optimal transportation.
\newblock {\em Acta Mathematica}, 177(2):113--161, 1996.

\bibitem{GittinsGlazebrook}
J.C. Gittins and K.D. Glazebrook.
\newblock On {B}ayesian models in stochastic scheduling.
\newblock {\em Journal of {A}pplied {P}robability}, 14:556--565, 1977.

\bibitem{Glazebrook}
K.D. Glazebrook, D.J. Hodge, and C.~Kirkbride.
\newblock General notions of indexability for queueing control and asset
  management.
\newblock {\em The Annals of Applied Probability}, 21(3):876--907, 2011.

\bibitem{HeHuang}
H.~He and C.F. Huang.
\newblock Consumption-portfolio policies: An inverse optimal problem.
\newblock {\em Journal of Economic Theory}, 62(2):257--293, 1994.

\bibitem{HobsonKlimmek:10}
D.~Hobson and M.~Klimmek.
\newblock Constructing time-homogeneous generalized diffusions consistent with
  optimal stopping values.
\newblock {\em Stochastics An International Journal of Probability and
  Stochastic Processes}, 83(4-6):477--503, 2011.

\bibitem{Jewitt}
I.~Jewitt.
\newblock {R}isk {A}version and the {C}hoice {B}etween {R}isky {P}rospects:
  {T}he {P}reservation of {C}omparative {S}tatics {R}esults.
\newblock {\em The {R}eview of {E}conomic {S}tudies}, 54(1):73--85, 1987.

\bibitem{Karatzas}
I.~Karatzas.
\newblock Gittins indices in the dynamic allocation problem for diffusion
  processes.
\newblock {\em The {A}nnals of {P}robability}, 12(1):173--192, 1984.

\bibitem{RK}
R.~Klimmek.
\newblock {\em Risikoneigung und {B}esteuerung}.
\newblock V. {F}lorentz, 1986.
\newblock {H}ochschulschriften zur {B}etriebswirtschafslehre, {B}d. 37.

\bibitem{BingLu}
B.~Lu.
\newblock Recovering a piecewise constant volatility from perpetual put option
  prices.
\newblock {\em Journal of Applied Probability}, 47(3):680--692, 2010.

\bibitem{Milgrom}
P.~Milgrom and I.~Segal.
\newblock Envelope theorems for arbitrary choice sets.
\newblock {\em Econometrica}, 70(2):583--601, 2002.

\bibitem{rogers}
L.C.G. Rogers and D.~Williams.
\newblock {\em Diffusions, Markov Processes and Martingales}, volume~2.
\newblock Cambridge University Press, 2000.

\bibitem{ruschfrech}
L.~R{\"{u}}schendorf.
\newblock Fr{\'{e}}chet-bounds and their applications.
\newblock In G.~Dall'Aglio et~al., editor, {\em Advances in Probability
  Distributions with Given Marginals : beyond the copulas}, pages 151--187.
  Kluwer {A}cademic {P}ublishers, 1991.

\bibitem{Salminen}
P.~Salminen.
\newblock Optimal stopping of {O}ne-{D}imensional {D}iffusions.
\newblock {\em Math. Nachrichten}, 124:85--101, 1985.

\bibitem{SamuelsonTax}
P.~Samuelson.
\newblock Tax deductibility of economic depreciation to insure invariant
  valuations.
\newblock {\em Journal of Political Economy}, 72(6):604--606, 1964.

\bibitem{SamuelsonPref}
P.A. Samuelson.
\newblock Consumption theory in terms of revealed preference.
\newblock {\em Economica}, pages 243--253, 1948.

\bibitem{Whittle}
P.~Whittle.
\newblock Restless bandits: Activity allocation in a changing world.
\newblock {\em Journal of {A}pplied {P}robability}, pages 287--298, 1988.

\end{thebibliography}

\end{document}